\documentclass[4apaper]{amsart} 
\usepackage[cp1251]{inputenc}
\usepackage[english]{babel}

\usepackage{amsmath,amsfonts,amssymb,mathrsfs,amscd,comment,latexsym,bbold}
\usepackage[matrix,arrow,curve]{xy}

\usepackage[usenames]{color}
\usepackage{colortbl}
\usepackage[dvipsnames]{xcolor}

\usepackage{enumitem}

\usepackage{csquotes}

\usepackage[top = 3cm, bottom = 4cm, left = 3cm, right = 3cm]{geometry}

\usepackage{amsthm}

\theoremstyle{plain}
\newtheorem{theorem}{Theorem}
\newtheorem*{nonum-theorem}{Theorem}

\newtheorem{lemma}{Lemma}
\newtheorem{lemma-remark}{Lemma-remark}

\newtheorem{corollary}{Corollary}
\newtheorem{nonum-corollary}{Corollary}
\newtheorem{conjecture}{Conjecture}

\newtheoremstyle{handleNumber}{}{}{\itshape}{}{}{}{\newline}{{\bf #1} \thmnote{#3}}
\theoremstyle{handleNumber}
\newtheorem*{handnum-theorem}{Theorem}

\theoremstyle{definition}

\newtheorem*{nonum-definition}{Definition}

\theoremstyle{remark}
\newtheorem{remark}{Remark}

\renewcommand{\leq}{\leqslant}
\renewcommand{\geq}{\geqslant}

\newcommand{\A}{\mathbb A}
\newcommand{\PP}{\mathbb P}

\newcommand{\SH}{\mathbf{SH}}

\newcommand{\Sch}{\mathrm{Sch}}

\newcommand{\Sm}{\mathrm{Sm}}

\newcommand{\Sh}{\mathrm{Sh}}
\newcommand{\SSh}{\mathrm{SSh}}

\newcommand{\Gm}{\mathbb G_m}

\newcommand{\codim}{\operatorname{codim}}

\newcommand{\nis}{\mathrm{nis}}

\newcommand{\EssSm}{\mathrm{EssSm}}

\newcommand{\calO}{\mathcal O}

\renewcommand{\Sh}{\mathbf{Sh}}

\renewcommand{\SSh}{\mathbf{SSh}}
\renewcommand{\H}{\mathbf{H}}
\renewcommand{\SH}{\mathbf{SH}}

\newcommand{\pt}{\mathrm{pt}}

\newcommand{\Set}{\mathrm{Set}}

\newcommand{\Fib}{\operatorname{Fib}}
\newcommand{\Cofib}{\operatorname{Cofib}}

\begin{document}

\title{Stable connectivity over a base}
\author{Druzhinin A.}
\thanks{The research is supported by the Russian Science Foundation grant 19-71-30002}
\address{Chebyshev Laboratory, St. Petersburg State University, 14th Line V.O., 29B, Saint Petersburg 199178 Russia}
\begin{abstract}

Morel's stable connectivity theorems 
state that
for any connective $S^1$-spectrum $F$ of motivic spaces (Nisnevich simplicial sheaves)
over an arbitrary field, the spectrum $L_{\A^1}(F)$ is connective, 
and the same property for $\PP^1$-spectra of motivic spaces.
Here $L_{\A^1}$ denotes the 
$\A^1$-localisation in the category of motivic spectra over a field $k$.

Originally the same property was conjectured for the case of motivic $S^1$-spectra over a base scheme $S$.
In view of Ayoub's conterexamples 
the modified version of conjecture states that $L_{\A^1}(F)$ is $(-d)$-connective for any connective $F$, where $d=\dim S$ is the Krull dimension.
The conjecture is proven under the infiniteness assumption on the residue fields 
for the cases of Dedekind schemes by J.~Schmidt and F.~Strunk  
and noetherian domains of arbitrary dimension by N.~Deshmukh, A.~Hogadi, G.~Kulkarni and S.~Yadavand. 
In the article we prove the result or general base with out the assumption on the residue fields.

So by the result for any smooth scheme $X$ over a base scheme $S$ of Krull dimension $d$
the Nisnevich sheaves of $S^1$-stable motivic homotopy groups $\pi_i^{S^1}(X)$ and $\PP^1$-stable motivic homotopy groups $\pi_{i+j,j}^{\PP^1}(X)$ vanishes for all $i<-d$. 
\end{abstract}

\subjclass{14F42, 55Q10}
\keywords{stable motivic homotopy groups, connectivity theorems, Gabber's presentation lemma.}
\maketitle
\section{Introduction}

\subsection{The stable connectivity theorem in $\SH(S)$ and $\SH_{S^1}(S)$.}
The Morel's connectivity theorems \cite[Theorem 4.2.10]{Mor0}, \cite[Theorem 6.1.8]{Mor1}, \cite[Theorem 18]{Mor2} 
states the vanishing of the negative motivic homotopy groups in negative degrees. This demonstrates the similar behaviour with the classical homotopy groups.
The first work \cite{Mor0} 
proves that for the 
sheaves of 
$S^1$-stable and $\PP^1$-stable motivic homotopy groups $\underline{\pi}^{S^1}_i(X)$ and $\underline{\pi}^{\PP^1}_{i+j,j}(X)$ vanishes for $i<0$
for any smooth scheme $X\in \Sm_k$ over a perfect base field $k$. 
In \cite{Mor1} the result is proven for an arbitrary base field.
In book \cite{Mor2} the result is proven in the unstable case; 
the proofs in the stable case are much easier. 

In the article we deal with stable motivic homotopy groups over a base scheme $S$.
As discussed later the conterexamples ware obtained by J.~Ayoub in \cite{Ayo06} shows that 
the Krull dimension $d=\dim S$ plays a role in this connectivity property, and $\A^1$-localisation can shift the homotopy t-structure by $-d$.
The positive results of such form for the stable motivic homotopy category over base schemes with infinite residue fields are proven in \cite{SS} by J.~Schmidt and F.~Strunk and 
\cite{DHKY} by N.~Deshmukh, A.~Hogadi, G.~Kulkarni and S.~Yadavand.

Let us point also that by definition sheaves of motivic homotopy groups are Nisnevich sheaves and it is essential for all mentioned results in the relative case, while the original arguments over fields prove the vanishing of the Zariski sheaves as well. 

Now we recall the notions of connective and $\A^1$-connective spectra of motivic spaces and the statement of the connectivity theorem in such terms.

Following the original definitions \cite{MV99}
a pointed motivic space $F$ over a base scheme $S$ is a Nisnevich simplicial sheaf on the category of smooth schemes over $S$.  A motivic $S^1$-spectrum is an $S^1$-spectrum of pointed motivic spaces. 
Denote by $\SSh^\bullet(S)$ the category of pointed motivic spaces over $S$, 
and by  $\SSh^{S^1}(S)$ the category of $S^1$-spectra of pointed motivic spaces. 
Equivalently we could say that $\SSh^{S^1}(S)$ is the category of Nisnevich sheaves of $S^1$-spectra. 

According to \cite{Mor1} 
an $S^1$-spectrum $E$ is called $n$-connective, if all the negative stable homotopy groups of $E\wedge S^n$ are trivial; and 
an $S^1$-spectrum of pointed motivic spaces $\mathcal F=(F_0, F_1, \dots F_i,\dots )$ is called \emph{$n$-connective},
iff for all $U\in \Sm_S$ the $S^1$-spectrum $\mathcal F(U)=(F_0(U), F_1(U), \dots F_i(U),\dots )$ is $n$-connective.
Precisely, the $S^1$-spectrum $\mathcal F$ is $n$-connective iff
\[\pi_i(\mathcal F(U))=0, \text{ for all } i<n, \text{ and all } U\in \Sm_S\] 
where $\pi_i$ are the classical motivic homotopy groups, so
\[\pi_i(\mathcal F(U))=[S^0, \mathcal F(U)\wedge S^i]_{\mathbf{SH}}= \varinjlim [S^j, \mathcal F(U)\wedge S^{i+j}]_{\mathbf H^\bullet},\]
and $\mathbf{SH}$ and $\mathbf H^\bullet$ denotes the (classical topological) stable homotopy category and pointed unstable homotopy category.

Denote by 
$L_{\A^1}\colon \SSh^{S^1}_\nis(S)\to \SSh^{S^1}_\nis(S)$ is the $\A^1$-localisation functor, 
that is the localisation functor with respect to the Bousfield localisation on $\SSh^{S^1}(S)$ 
generated by the equivalences of the form $U\times \A^1\to U$. 
Then \emph{$\A^1$-$n$-connective} $S^1$-spectrum of pointed motivic spaces $\mathcal F=(F_0, F_1, \dots F_i,\dots )$
is a spectrum such that $L_{\A}(\mathcal F)$ is n-connective. An {$\A^1$-$0$-connective} spectrum is called by \emph{$\A^1$-connective}.


Morel's stable connectivity theorem states the following 
\begin{theorem}[Theorem 6.1.8 in \cite{Mor1}] \label{th:StConnTh(k)}
If an $S^1$-spectrum of Nisnevich sheaves $\mathcal F\in \SSh^{S^1}(k)$ is connective
for an arbitrary base field $k$, then $\mathcal F$ is $\A^1$-connective, i.e.
$L_{\A}(\mathcal F)$ is 0-connective spectrum of motivic spaces.
\end{theorem}
\begin{remark}
Equivalently the statement can be given in terms of homotopy t-structure, with positive truncations being connective spectra.
See \cite[Theorem 5.2.3]{Mor0} for the similar statement in the case of $\PP^1$-spectra.
\end{remark}
\begin{remark}
Since for any $X\in \Sm_k$ the suspension spectra $\Sigma^\infty_{S^1} X$ and $\Sigma^\infty_{\PP^1} X$ are connective, it follows that they are $\A^1$-connective, and so
the Nisnevich sheaves of stable motivic homotopy groups $\underline{\pi}^{S^1}_i(X)$ and $\underline{\pi}^{\PP^1}_{i+j,j}(X)$ are trivial for all $i<0$.
\end{remark}


The original Morel's stable connectivity conjecture claims that the same property holds in $\SH_{S^1}(S)$ for an arbitrary base scheme $S$.
In \cite{Ayo06} Ayoub had constructed 
conterexamples given by pointed motivic spaces $F$ such that $L_{\A^1}(F)$ is not 0-connective, but it is $(-d)$-connective, $d=\dim S$. 
The modified version of Morel's conjecture in view of Ayoub's conterexamples is the following statement: 
\begin{conjecture}\label{con:Conn-dim=d}
Let $S$ be a base scheme, $\dim S=d$. 
If an $S^1$-spectrum of Nisnevich sheaves  
$\mathcal F\in \Sh_{\nis}^{S^1}(S)$ 
is 
$0$-connective, then 
the spectra 
$L_{\A^1}(\mathcal F)$ is $(-d)$-connective,
where $\Sh_\nis^{S^1}(S)$ denotes the category of $S^1$-spectra of Nisnevich sheaves  over $S$.
\end{conjecture}

The connectivity theorem in this form was proven in \cite{SS} by J.~Schmidt and F.~Strunk for the case of Dedekind base schemes with infinite residue fields.
In \cite{DHKY} the result is extended by N.~Deshmukh, A.~Hogadi, G.~Kulkarni and S.~Yadav to the case of arbitrary noetherian domains with infinite residue fields.
In the article we prove the result for arbitrary base scheme of finite Krull dimension with arbitrary residue fields. So the result is the theorem 
\begin{theorem}
Let $S$ be a scheme of Krull dimension $d$.
Let $U$ be an essentially smooth local henselian scheme over a base scheme $S$.
Let $F\in \SH_{S^1}(S)$ be a connective spectrum.
Then \[[U, F\wedge S^i]_{\SH_{S^1}(S)}=0, [U, \Sigma_{\Gm}^\infty F\wedge S^i]_{\SH(S)}=0 \text{, for all } i>d.\]
\end{theorem}

In \cite{SS} and \cite{DHKY} the connectivity theorem is proven following the original reasoning with proving of the 
the corresponding versions of Gabber's presentation lemma. 
In the present work we don't use and prove the Gabber's presentation lemma over a base in general,
but we follow the same kind of arguments as in the original reasoning \cite{Mor1} to prove the connectivity theorem.
Actually we could say that implicitly we prove the Gabber's presentation lemma for a local henselian scheme over a local henselian base, see remark \ref{rem:HensRelGabber'sPresLm} and the proof of theorem \ref{th:Injectivity}.

\subsection{The proof of stable connectivity over a field.} 
The proof of the stable connectivity theorem over a field, see theorem \ref{th:StConnTh(k)}, consists in short of two steps:
\begin{theorem}[Lemma 4.1.6, \cite{Mor1}]\label{th:InjExactnessX/UContr}
Let $X$ be a smooth scheme over a field $k$. Let $Z\subset X$ be a closed subscheme of positive codimension.
Then the motivic space $X/(X-Z)$ is $0$-connective.
\end{theorem}
\begin{theorem}[Theorem 4.14 \cite{VoevCongress}, Lemma 4.3.1 \cite{Mor1}, Proposition 3.1 \cite{SS}]\label{th:Vanishing}
Let $U\in \Sm_S$ over a base scheme $S$, and $E$ be a connective $S^1$-spectrum of Nisnevich sheaves. 
Then \begin{equation}\label{eq:Vanishing} [U, E\wedge S^i]_{\SH_{S^1}(S)}=0, \text{ for all } i>\dim U.\end{equation}
\end{theorem}
The first theorem above is equivalent to the exactness of the sequence of morphisms of pointed sets
\begin{equation}\label{eq:InjexactnessH}*\to [U,F]_{\mathbf H(S)}\hookrightarrow [U-Z,F]_{\mathbf H(S)}, \forall U\in \EssSm_S \text{ local henselian}, \text{ and }\forall F\in \mathbf H_\bullet(S).\end{equation} 
and implies the \emph{injectivity}
of the homomorphism of $S^1$-stable (and $\PP^1$-stable) motivic homotopy groups
\begin{equation}\label{eq:InjectivitySH}[U,F]_{\SH_{S^1}(S)}\hookrightarrow [U-Z,F]_{\SH_{S^1}(S)}, \forall U\in \EssSm_S \text{ local henselian}, \text{ and }\forall F\in \SH_{S^1}(S).\end{equation} 
and similarly for the $\PP^1$-stable category $\SH(S)$.

By arguments of \cite{Mor0} the connectivity theorem follows, since if $\eta\in U$ is the generic point, then
\begin{gather*}
[U, Y\wedge S^i]_{\SH_{S^1}(k)} \stackrel{\eqref{eq:InjectivitySH}}{\hookrightarrow} [\eta, Y\wedge S^i]_{\SH_{S^1}(k)}\stackrel{\eqref{eq:Vanishing}}{=} 0, \text{ for all } i>0, \text{ and }\\ 
[U, Y\wedge S^i]_{\SH(k)} = \varinjlim\limits_n [\Sigma^n_{\PP^1}U, (\Sigma_{\Gm}^n Y)\wedge S^{i+n} ]_{\SH_{S^1}(k)} \stackrel{\eqref{eq:InjectivitySH}}{\hookrightarrow} [\eta\wedge (\PP^1)^{\wedge n}, (\Sigma_{\Gm}^n Y)\wedge S^{i+n} ]_{\SH_{S^1}(k)}\stackrel{\eqref{eq:Vanishing}}{=} 0,
\end{gather*}
since $\dim\eta=0$, and $\dim \PP^1_\eta=1$, $\dim \infty_\eta=0$, where $\infty_\eta\subset \PP^1_{\eta}$ is the infinity point.

\subsection{The vanishing theorem (over a base).}
The original argument for \emph{vanishing theorem} \eqref{eq:Vanishing}
in \cite{Mor1} and \cite{Mor0} 
is based on the "geometric" model for the stable $\A^1$-localisation functor $L_{\A^1}^{\nis,S^1}$ on the category $\SH^{S^1}_\nis(S)$ of $S^1$-spectra of Nisnevich sheaves, that leads to the isomorphism
\[ [U, E ]_{\SH_{S^1}(S)}\stackrel{def}{=}[U, L_{\A^1}^{\nis,S^1} E ]_{\SH^{S^1}_{\nis}(S)}\simeq \varinjlim_n [ U\wedge (\A^1/\{0,1\})^n, E ]_{\SH^{S^1}_{\nis}(S) }. \]
Then form the case of $E=\mathrm{EM}(M)$ being Eilenberg-Maclane spectrum of a Nisnevich sheaf of abelian groups $M$, the vanishing \eqref{eq:Vanishing} follows because of vanishing of Nisnevich cohomologies $H^i_\nis(X,M)=0$ for all $i>\dim X$. 
Then analysis of the Postnikov tower leads to 
\[[U, E\wedge S^i]_{\SH^{S^1}_{\nis}(S)} = 0, \text{ for all } U\in \Sch_S, i>\dim U, E\in \SH^{S^1}_{\nis,\geq 0}(\Sm_S)\]
and implies \eqref{eq:Vanishing} in the case of general connective spectrum of sheaves $E$.
The complete analysis of the Postnikov towers and the proof of the vanishing theorem in the general relative case is given in \cite[Section 3]{SS}.


\subsection{Proof of the injectivity property over a field.}
The original proof of theorem \ref{th:InjExactnessX/UContr} 
is given by the independent arguments for the perfect base field case \cite{Mor0} and the infinite base field case \cite{Mor1}. 

In the perfect case, for any closed $Z\subset U$ there is a filtration $Z_0\subset \dots\subset Z_{\dim Z}=Z$ such that $Z_{i+1}-Z_i$ is smooth. This allows to reduce the question to the case of smooth $Z$ be means of by purity isomorphism $U/(U-Z)\simeq Th(T_{Z/U})$.

The argument for the infinite field case is based on Gabber's presentation lemma.
The Gabber's presentation lemma over finite fields was proven later by A.~Hogadi and G.~Kulkarni in \cite{HK}, 
by this the second argument had covered the case of an arbitrary base field.

Let us point out hat also one universal proof of the theorem \ref{th:InjExactnessX/UContr} 
in $\H(k)$ over an arbitrary field follows from the constructions obtained in Panin's work \cite{Pan-GSConj} on the Grothendieck--Serre conjecture. 

\subsection{Proof of the connectivity theorem and the injectivity theorem over a base.}

In general relative case over a base scheme $S$ 
the vanishing theorem \eqref{eq:Vanishing} holds, and the injectivity \eqref{eq:InjectivitySH} fails.

\begin{nonum-theorem}[Injectivity over a base]
The sequence \eqref{eq:InjexactnessH} is exact
for a local (henselian) base scheme $S$, with closed point $\sigma\in S$,
and an essentially smooth local henselian scheme $U$, and a closed subset $Z\subset U$ such that \begin{equation}\label{eq:codimZsimga>0}\codim_{U_\sigma} Z_\sigma>0,\; U_\sigma=U\times_S \sigma, Z_\sigma=Z\times_S \sigma, \sigma\in S,\end{equation}
in other words the closed fibre $Z_\sigma$ of $Z$ is of positive codimension in the closed fibre $U_\sigma$ of $U$.
\end{nonum-theorem}

Then connectivity theorem \eqref{con:Conn-dim=d} for the case of a local $S$ with closed point $\sigma\in S$ follows, since
\[ [U, Y\wedge S^i]_{\SH_{S^1}(S)} \stackrel{\eqref{eq:InjectivitySH}}{\hookrightarrow} [U_\eta, Y\wedge S^i]_{\SH_{S^1}(S)}\stackrel{\eqref{eq:Vanishing}}{=} 0, \forall i>\dim U, 
\]
where $\eta\in U_\sigma$ is the generic point, 
since $\dim U_\eta=\codim_U U_\sigma=\codim_S \sigma\leq \dim S$.
The general case of a base scheme $S$ follows since for a local henselian $U$ we have 
\[[U, Y\wedge S^i]_{\SH_{S^1}(S)}=[U, (Y\times_S S^\prime)\wedge S^i]_{\SH_{S^1}(S^\prime)},\]
where $S^\prime$ is the local scheme of $S$ at the point that is the image of the closed point of $U$.


%

\subsection{Acknowledgement} 
The acknowledgement for the Hakon Kolderup for many of helpful remarks on the article.
The author is grateful for the hospitality for the Mathematical Department of the University of Oslo and the Frontier research group project "Motivic Hopf Equations".

\section{Injectivity theorem}

\begin{theorem}\label{th:Injectivity}
Let $S$ be a local henselian scheme, and let $\sigma\in S$ be the closed point.
Let $X\in Sm_S$, and $x\in X$, and $Y\subset X$ be a closed subscheme $x\in Z$, $Y$ is of positive relative codimension over $S$ at $x$.
Let $U=X^h_x$ and $Z=Y^h_x$ are the henselizations. 


Then the class of the canonical morphism of pointed sheaves $U_+\to U/(U-Z)$ is equal to the class of the pointed morphism in $[U,U/(U-Z)]_{\H_\bullet(S)}$. 
\end{theorem}

\begin{corollary}
Let $S$, $U$, $Z$ be are as in the theorem \ref{th:Injectivity}.

Then for any functor $F\colon \H_\bullet(S)\to \Set_\bullet$ such that for any morphism $f\colon \mathcal X_1\to\mathcal X_2$ in $\H_\bullet(S)$ the sequence $F(\Cofib(f))\to F(\mathcal X_1)\to F(\mathcal X_2)$ is exact at the middle term, the fibre of the map $F(U)\to F(U-Z\times_X U)$ is trivial.

\end{corollary}
\begin{corollary}
Let $S$, $U$, $Z$ be are as in the theorem \ref{th:Injectivity}.

Then for any pointed motivic space $(\mathcal X,t)$ the fiber of the restriction homomorphism $\Fib(\pi_0^{\A^1}(\mathcal X,t)(U)\to \pi_0^{A^1}(\mathcal X,t)(U-Z))$ is trivial, where $\pi_0^{\A^1}(\mathcal X)=[-, \mathcal X]_{\H_{\bullet}(S)}$ denotes the pointed set of unstable motivic homotopy groups of $\mathcal X$. 
\end{corollary}
\begin{corollary}
Let $S$, $U$, $Z$ be are as in the theorem \ref{th:Injectivity}.

Then for any $S^1$-spectra (or $\PP^1$-spectra) of motivic spaces $E$ the restriction homomorphism $\pi_i^{S^1}(E)(U)\to \pi_i^{S^1}(E)(U-Z)$ (or $\pi_i^{\PP^1}(E)(U)\to \pi_i^{\PP^1}(E)(U-Z)$) is trivial, where $\pi_i^{S^1}(E)=[-, E]_{\SH_{S^1}(S)}$ ($\pi_i^{\PP^1}(E)=[-, E]_{\SH_{\PP^1}(S)}$) denotes the presheaf of $S^1$-stable ($\PP^1$-stable) motivic homotopy groups of $E$. 
\end{corollary}

\begin{proof}

%
%

Let $x\in U$ be the closed point.
Now let us point that since $U$ is local henselian, it follows that $Z$ is either empty or a local henselian scheme with the same closed point $x$. The case of empty $Z$ is trivial, so we can assume that $Z\neq \emptyset$.

Firstly, we reduce the question to the case of a scheme $U$ of relative dimension one over $S$.
By assumption the subscheme $Z\times_S \sigma$ is of positive codimension in $U\times_S \sigma$,
and $U$ is essentially smooth
Hence by lemma \ref{lm:genericMaptoA^(n-1)} below there is a map $p\colon U\to \A^{n-1}_S$ where $n=\dim_S U$,
such that $p$ is (essentially) smooth, and $Z\times_{\A^{n-1}_S} p(x)$ is a closed subscheme of positive codimension in $\A^{n-1}_S$.
So we can redefine $S$ as the henselization of $\A^n_S$ at $p(x)$ and assume by this that $U$ is an essentially smooth local henselian  scheme over $S$.

Since $U$ now is of relative dimension one over $S$, and $Z\times_S \sigma$ is of positive codimension in $U\times_S \sigma$, it follows that $Z$ is quasi-finite over $S$. 
Then since $S$ is local henselian, and $Z$ is local by the above, it follows by lemma \ref{lm:HensFin-quasifin} below that $Z$ is finite over $S$.

By lemma \ref{lm:etalePres} there is an etale map $\pi\colon U\to \A^1_S$ that induces an isomorphism $x\simeq \pi(x)$.
Consider the image $\pi(Z)\subset \A^1_S$, which is a closed subscheme in $\A^1_S$, since $Z$ is finite over $S$.
It follows that $\pi(Z)$ is finite over $S$. 
Since $U$ is local it follows that $Z$ and $\pi(Z)$ have one connected component; and  since $S$ is local henselian, it follows that
$Z$ and $\pi(Z)$ are local henselian.

Now we see that the induced morphism $Z\to \pi(Z)$ is a finite, \'etale morphism of local henselian schemes, and it is an isomorphism on closed points; thus $\pi$ induces an isomorphism $Z\to \pi(Z)$ by Nakayama's lemma. 

Now we argue like in \cite[Lemma 4.1.6]{Mor1}.
Since $\pi(Z)\subset \A^1_S$ is a closed subscheme finite over $S$, it is a closed subscheme in $\PP^1_S$ that does not meet the infinity section $\infty_S\subset \PP^1_S$.
Consider the sequence of equivalences in $\H(S)$:
\begin{equation}\label{eq:U/(U-Z)=PP1/(PP1-pi(Z))}U/(U-Z)\simeq \A^1_S/(\A^1_S-\pi(Z))\simeq \PP^1_S/(\PP^1_S-\pi(Z))\end{equation}
Looking at the following diagram we see that the class of the canonical morphism $U\to \PP^1_S/(\PP^1_S-\pi(Z))$ in $\H(S)$ is equal to the class of the morphism $U\to S\xrightarrow{\infty} \PP^1_S\to \PP^1_S/(\PP^1_S-\pi(Z))$
\[\xymatrix{ U \ar[r]\ar[d]&\A^1_S\ar[r]^{\simeq}\ar[dl]_j & \pt_S\ar[dll]^{\infty}\ar[dl]^{*} \\
\PP^1_S \ar[r]& \PP^1_S/(\PP^1_S-\pi(Z))  ,}\]
where 
$\infty\colon \pt_S\to \PP^1_S$ denotes the infinity section, 
and $j$ is the immersion $\A^1\simeq \PP^1_S-\infty_S\to \PP^1_S$.
Since $\infty_S\subset\PP^1_S-\pi(Z)$, the composition right diagonal morphism $\pt_S\to \PP^1_S/(\PP^1_S-\pi(Z)) $ is equal to the pointed one in the category $\H(S)$.
Thus it follows that the class of the morphism $U_+\to \PP^1_S/(\PP^1_S-\pi(Z))$ in the category $\H_\bullet(S)$ is equal to the pointed one.
Hence by the equivalences \eqref{eq:U/(U-Z)=PP1/(PP1-pi(Z))} the class of the canonical morphism $U_+\to U/(U-Z)$ is equal to the pointed morphism in the category $\H_\bullet(S)$.
The claim follows.
\end{proof}

\begin{remark}\label{rem:HensRelGabber'sPresLm}
By the above arguments we have proven Gabber's presentation lemma for henselian local essentially smooth schemes:

Let $U$ be an essentially smooth henselian local scheme over a scheme $S$ and let $Z\subset U$ be a closed subscheme of positive relative codimension, $\codim_{U/S} Z>0$. Then there is a map 
$\pi\colon U\to \A^1_{S^\prime}$, where $S^\prime = (\A^{\dim U-1})^h_0$ is the henselisation at the point 0, 
$\pi$ is etale, $\pi$ induces an isomorphism $Z\simeq \pi(Z)$, and $\pi(Z)$ is finite over $S^\prime$.
\end{remark}

\begin{lemma}\label{lm:etalePres}
Let $U$ be an essentially smooth local henselian scheme over a base scheme $S$.
Then there is a (pro-)etale map $\pi\colon U\to \A^n_S$ that induces an isomorphism $x\to \pi(x)$, where $x\in U$ is the closed point, and $n=\dim_S U$.
\end{lemma}
\begin{proof}
Firstly, we show that the relative case follows form the base field case.
Let $\pi^\prime\colon U\times_S \sigma\to \A^n_\sigma$ be the required morphism over the closed point $\sigma\in S$.
Assume that there is an (pro-)etale map $\pi^\prime\colon U\times_S \sigma\to \A^n_\sigma$ such that $x\simeq \pi^\prime(x)$.
Let $\pi$ be any lift of $\pi^\prime$.
Then it follows that $\pi$ is etale and $x\simeq \pi(x)$.

In the base filed case the claim is well known, 
It is particular case of Gabber's presentation lemma \cite{Gab},\cite{CTHK} and \cite{HK}, but the proof of the required claim is much shorter.
Let us recall the argument.

Let $k$ be the base field.
Now we use the fact that any (finite) separable extension of fields is primitive. 
If follows that $k(x)=k(\alpha)$ for some $\alpha\in k(x)$. 
Consider the regular map $(\alpha,0,\dots, 0)\colon x\to \A^n_k$.
Then any lift $\pi\colon U\to \A^n_k$ of $(\alpha,0,\dots, 0)\colon x\to \A^n_k$ induces an isomorphism $x\simeq \pi(x)$. 
Choose a trivialisation $\tau\colon T_{U,x}\simeq \mathbb 1^n_{k(x)}$ of the tangent space at the point $x$.
Hence for any given $\pi$ we have the isomorphism $T_{U,x}\to T_{\A^n,\pi(x)}\simeq \mathbb 1_{k(x)}$.
Define $\pi\colon U\to \A^n_k$ as such a lift of $(\alpha,0,\dots ,0)$ that the differential morphism 
$\pi\big|_{T_{U,x}}\colon T_{U,x}\to \pi^*(T_{\A^n,\pi(x)})$ is equal to the composition of $\tau$ and the inverse image of the canonical isomorphism $\mathbb 1_{k(x)}\simeq T_{\A^n,\pi(x)}$.

Now consider the case of the infinite base field $k$.
Let $\overline k$ be the algebraical closure.
Let $X$ be affine $k$-scheme such that $U$ is isomorphic to local scheme of $X$ at $x$.
Let $X\to \A^N_k$ be a closed immersion. 
Consider the affine space $\Gamma$ of the linear projections $\A^N_k\to \A^n_k$. Then the set of point of $\Gamma$ that parametrises the set of the projections that do not satisfy the claimed properties is the image of the similar set $B$ in $\Gamma_{\overline k}$.
Then we see that the set $B\subset \Gamma_{\overline k}$ is the proper closed subset. Finally we use that any non-empty open subset of the affine space over an infinite field has a rational point.
\end{proof}

\begin{lemma}\label{lm:genericMaptoA^(n-1)}
Let $S$ be a scheme.
Let $X$ be a smooth scheme of relative dimension $n$ over $S$, 
and let $x\in X$ be a point over a point $\sigma\in S$.
Let $Z\subset X$,$Z\ni x$, be a closed subscheme of positive codimension over $\sigma$.

Then there is a regular map $f=(f_1,\dots f_{n-1})\colon X\to \A^{n-1}_S$ smooth at $x$ and such that $X$ is of relative dimension $1$ over $\A^{n-1}_S$, and
$Z$ is of positive codimension over $p(x)\in \A^{n-1}_S$.
\end{lemma}
\begin{proof}
Without loss of generality we can assume that $x\in X$ is a closed point in a scheme of finite type $X$, 
and assume that $Z\subset X$ is a closed subscheme such that $Z\times_S \sigma$ of pure codimension one in $X\times_S\sigma$.

Since $x$ is a smooth point it follows that there is a trivialisation $\tau=(\tau_1,\dots ,\tau_{n})\colon \calO(x)^n\simeq \Omega_{x/X}$, where $\calO(x)$ denotes the residue field at $x$.
Denote by $\Gamma$ the subset in $\mathcal O(X)^{n-1}$ that consists of sets $f=(f_1,\dots, f_{n_1})$ such that the differential of $f$ is equal to $\tau$, i.e. $\mathrm d f_i=\tau_i$, $i=1,\dots, {n-1}$.
Then for any  $f\in \Gamma$ the morphism of schemes $f\colon X\to \A^{n-1}_S$ is smooth at $x$.

We are going to construct by induction on $i$ the functions $f_1,\dots f_{i}\colon X\to \A^1_S$ such that 
$f_j\big|_{x}=0$, $\mathrm d f_j=\tau_j$, $ j=1,\dots, i$, and such that 
$Z$ is of a positive codimension over the point $0$ that is
the image $f^i(x)\in \A^i_S$ of $x\in X$ under the morphism $f^i=(f_1,\dots ,f_i)\colon X\to \A^i_S$, and
such that the dimension of the morphism $f^i$ is $n-i$.

Assume that we have constructed the functions $f_1,\dots, f_i$, $i<n-1$. We are going to construct $f_{i+1}$.
The base is the case of $i=0$.

Consider the map $(f_1,\dots f_i)\colon X\to \A^i_S$. 
Let \[W_i\stackrel{\mathrm{def}}{=}Z\times_{\A^i_S} 0, F_i\stackrel{\mathrm{def}}{=}Z(f_1,\dots f_i)=X\times_{\A^i_S} 0\subset X.\] 
Then, by assumption, $F_i$ is of pure codimension $i$ in $X\times_S \sigma$, and $W_i$ is a closed subscheme in $F_i$ of pure codimension one. 
Hence  $W_i$ is of pure codimension $1+i$ in $X$.

Choose a set of closed points $C\subset W_i$ such that $x\not\in C$ and $C$ contains at least one point in each irreducible component of $W_i$.
Now choose $f_{i+1}$ in such a way that \[f_{i+1}\big|_{x}=0, f_{i+1}\big|_{Z(I^2(x))}=\tau_{i+1}, f_i\big|_{S}=1.\]
We need to prove that the relative dimension of the map $f^{i+1}=(f_1,\dots, f_{i+1})\colon X\to \A^{i+1}_S$ is $n-{i+1}$, and the relative codimension of $Z$ in $X$ is positive over the point $0\in \A^{i+1}_S$.
In other words we need to prove that $\codim_{X\times_S \sigma} F_{i+1}={i+1}$, and $\codim_{F_{i+1}} W_{i+1} =1$, where
$W_{i+1}\stackrel{def}{=}Z\times_{\A^{i+1}_S} 0$, $F_{i+1}\stackrel{def}{=}X\times_{\A^{i+1}_S} 0\subset X$.

By assumption it follows that $f_i$ does not vanish on any irreducible component of $W_i$, 
and hence $F_{i+1}$ is of positive codimension in $F_{i}$, and
$W_{i+1}$ is of positive relative codimension in $W_i$. 
So $\codim_{X\times_S \sigma} F_{i+1}= \codim_X F_i+1= i+1, \codim_{X\times_S \sigma} W_{i+1}= \codim_{X\times_S \sigma} W_{i}+1 = i+2$.  
So the first claim is done, and since we see that $\codim_{F_{i+1}} W_{i+1}= \codim_{X\times_S\sigma} W_{i+1}- \codim_{X\times_S\sigma} F_{i+1}=1$ the second claim follows too.
\end{proof}
\begin{lemma}\label{lm:HensFin-quasifin}
Let $Z$ be a quasi-finite local scheme over a henselian local scheme $S$.
Then $Z$ is finite over $S$.
\end{lemma}
\begin{proof}
By Zariski's main theorem \cite[Theorem~8.12.6]{GD67} 
there is a scheme $\overline Z$ finite over $S$ and a dense open immersion $Z\to \overline Z$.
We are going to prove that $Z=\overline Z$.
Since $\overline Z$ is finite over $S$, by Nakayama's lemma the claim is equivalent to 
$(\overline Z\setminus Z)\times_S \sigma = \emptyset$, where $\sigma\in S$ is the closed point.

Assume that $x\in (\overline Z\setminus Z)\times_S \sigma$.
Then $x$ is not equal to the closed point of $Z$ and hence the closed fibre $\overline Z\times_S \sigma$ splits as $\overline Z\times_S \sigma= D_1 \amalg D_2$.
Since $S$ is henselian it follows that 
\[\overline Z= E_1\amalg E_2,
x\in E_1, x\not\in X_2.
\]
But since $Z$ is local by assumption it follows that $\overline Z$ is irreducible.
\end{proof}

\section{Vanishing theorem}

Firstly we recall the so called $S^1$-stable Vanishing theorem. The result is known due to \cite[Theorem 4.14 ]{VoevCongress} in the case of the Spanier-Whitehead category of motivic spaces,
and in the case of the $S^1$-stable motivic homotopy category the theorem is known due to and proven in \cite{Mor1} and \cite{SS}.
We refer the reader to the mentioned sources for detailed arguments, here we present only a sketch. 
\begin{theorem}[Proposition 3.1. \cite{SS}, Lemma 4.2.1 \cite{Mor1}]\label{th:Vanishing}
Let $S$ be a base scheme.
Let $U\in \Sm_S$ be a scheme of Krull dimension $d$.
Let $F\in \SH_{S^1}(S)$ be a 0-connective spectrum.
Then $[U, F\wedge S^i]_{\SH_{S^1}(S)}=0$ for all $i>d$.
\end{theorem}
\begin{proof}[Sketch of the proof]
Firstly let us recall the formula 
for the $\A^1$-localisation functor $L_{\A^1}^{\nis,S^1}$ on the category $\SH^{S^1}_\nis(S)$ of $S^1$-spectra of Nisnevich sheaves:
\[L^{\nis,S^1}_{\A^1}(F)= (\Omega^\infty_{S^1_{\A}}\Sigma^\infty_{S^1} F, \dots ,\Omega^\infty_{S^1_{\A}}\Sigma^{\infty+i}_{S^1} F,\dots), S^1_{\A}=\Delta^1_{\A^1}/\Delta^0_{\A^1}=\A^1/\{0,1\},\]
and precisely the isomorphism 
\[ [U, E ]_{\SH_{S^1}(S)}\stackrel{def}{=}[U, L_{\A^1}^{\nis,S^1} E ]_{\SH^{S^1}_{\nis}(S)}\simeq \varinjlim_n [ U\wedge (\A^1/\{0,1\})^n, E ]_{\SH^{S^1}_{\nis}(S) }. \]
Briefly speaking the formula and the isomorphism follows since $S^1_{\A}=\A^1/\{0,1\}$ is the $\A^1$-homotopy equivalent model for the sphere $S^1$, and since the functor $\Omega^\infty_{S^1_{\A}}$ (and $\Sigma^\infty_{S^1}$)
preserves Nisnevich local objects in the category of $S^1$-spectra of Nisnevich sheaves.

By using the Postnikov tower the case of connective spectra of Nisnevich sheaves $E$ is equivalent to the case of the Eilenberg-Maclane spectrum $E=\mathrm{EM}(M)$ of an abelian Nisnevich sheaf $M$. 
Finally, since we have vanishing of the Nisnevich cohomologies
\[H^i_\nis(X,M)=0, \forall i>\dim X,\]
the claim follows.

We refer the reader to 
\cite[Section 4.3]{Mor1} and 
\cite[Section 3]{SS} for more details.
\end{proof}

\section{Connectivity theorem}
\begin{theorem}\label{th:Connectivity}
Let $S$ be a scheme of Krull dimension $d$.
Let $U$ be an essentially smooth local henselian scheme over a base scheme $S$.
Let $F\in \SH_{S^1}(S)$ be a 0-connective spectrum.
Then \[[U, F\wedge S^i]_{\SH_{S^1}(S)}=0, [U, F\wedge S^i]_{\SH(S)}=0 \forall i>d.\]
\end{theorem}
\begin{proof}
Without loss of generality we can assume that $S$ is local henselian since $U$ is such.
Denote by $E$ the set of generic points of $U\times_S \sigma$, where $\sigma\in S$ is the closed point.
Consider the semi-local essentially smooth scheme $V\subset U$:
\[V = \varprojlim_{Z\subset U, Z\cap E=\emptyset} (U-Z),\]
where the limit is over the set of all closed subschemes $Z$ in $U$ that do not meet $E$.
In other words $V$ is the complement to the closed subschemes in $U$ of positive relative codimension over $\sigma$.
Then since the codimension of $\sigma$ in $S$ is less then or equal to $d$, the dimension of the scheme $V$ is less then or equal to $d$.
Hence by theorem \ref{th:Vanishing} it follows that 
\[[V, F\wedge S^i]_{\SH_{S^1}(S)}=0, \forall i>d, \text{ and } [V\wedge (\PP^1/\infty)^{\wedge l}, F\wedge S^i]_{\SH_{S^1}(S)}=0, \forall i>d+l,\]
and hence
\[[V, F\wedge S^i]_{\SH_{S^1}(S)}=0, [V, F\wedge S^i]_{\SH(S)}=0, \forall i>d.\]
Let us note that in the right-hand side equalities we use the idea from \cite{Mor0} to use the cohomological dimension of $\PP^1$ to control the t-structure shift of the smash-product with motivic spheres $\PP^1$ and $\Gm$.
Finally, the claim follows, since by theorem \ref{th:Injectivity}
the canonical maps
\[ [U, F\wedge S^i]_{\SH_{S^1}(S)} \to [V, F\wedge S^i]_{\SH_{S^1}(S)}, [U, F\wedge S^i]_{\SH_{S^1}(S)} \to [V, F\wedge S^i]_{\SH(S)}\]
are injective.
\end{proof}

\end{document}